\documentclass[review]{elsarticle}

\usepackage{lineno,hyperref}
\usepackage{amsmath,amssymb}
\usepackage{xcolor}
\usepackage{amsthm} 
\usepackage{enumitem}
\usepackage{multicol}

\let\OLDthebibliography\thebibliography
\renewcommand\thebibliography[1]{
  \OLDthebibliography{#1}
  \setlength{\parskip}{0pt}
  \setlength{\itemsep}{0pt}
}

\allowdisplaybreaks
\modulolinenumbers[5]

\newcommand{\spleq}[2]{\begin{equation}\label{#1}\begin{split}#2\end{split}\end{equation}}
\newcommand{\spleqno}[1]{\begin{equation*}\begin{split}#1\end{split}\end{equation*}}

\newcommand{\para}[1]{\vspace{3mm} \noindent\textbf{#1.}} 
\newtheorem{The}{Theorem}[section]
\newtheorem{Lem}[The]{Lemma}

\newtheorem{Rem}[The]{Remark} 
\newtheorem{Def}[The]{Definition}

\journal{arXiv}

\begin{document}

\begin{frontmatter}

\title{An inverse problem for the fractional Schr\"odinger equation in a magnetic field}

\author{Giovanni Covi}
\address{Department of Mathematics, University of Jyv\"askyl\"a, Finland}
\address{\emph{\texttt{giovanni.g.covi@jyu.fi}}}

\begin{abstract}
This paper shows global uniqueness in an inverse problem for a fractional magnetic Schr\"odinger equation (FMSE): an unknown electromagnetic field in a bounded domain is uniquely determined up to a natural gauge by infinitely many measurements of solutions taken in arbitrary open subsets of the exterior. The proof is based on Alessandrini's identity and the Runge approximation property, thus generalizing some previous works on the fractional Laplacian. Moreover, we show with a simple model that the FMSE relates to a long jump random walk with weights.
\end{abstract}

\begin{keyword}
Fractional magnetic Schr\"odinger equation \sep Non-local operators \sep Inverse problems \sep Calder\'on problem 
\MSC[2010] 35R11, 35R30
\end{keyword}

\end{frontmatter}

\section{Introduction} 

This paper studies a fractional version of the Schr\"odinger equation in a magnetic field, or a fractional magnetic Schr\"odinger equation (FMSE), establishing a uniqueness result for a related inverse problem. We thus deal with a non-local counterpart of the classical magnetic Schr\"odinger equation (MSE) (see \cite{NSU95}), which requires to find up to gauge the scalar and vector potentials existing in a medium from voltage and current measurements on its boundary.

Let $\Omega\subset\mathbb R^n$ be a bounded open set with Lipschitz boundary, representing a medium containing an unknown electromagnetic field. The solution of the Dirichlet problem for the MSE is a function $u$ satisfying 
$$\left\{\begin{array}{lr}
        (-\Delta)_A u + qu := -\Delta u -i \nabla\cdot(Au) -i A\cdot\nabla u + (|A|^2+q)u =0 & \text{in } \Omega\\
        u=f & \text{on } \partial\Omega
        \end{array}\right. \;, $$

\noindent where $f$ is the prescribed boundary value and $A,q$ are the vector and scalar potentials in the medium. The boundary measurements are encoded in $\Lambda_{A,q} : H^{1/2}(\partial\Omega)\rightarrow H^{-1/2}(\partial\Omega)\;,$ the Dirichlet-to-Neumann (or DN) map. The inverse problem consists in finding $A, q$ in $\Omega$ up to gauge by knowing $\Lambda_{A,q}$.

The study of the local MSE has both mathematical and practical interest, since it constitutes a substantial generalization of the Calder\'on problem (see \cite{Ca80}). This problem first arose for the prospection of the ground in search of valuable minerals. In the method known as Electrical Impedance Tomography (EIT), electrodes are placed on the ground in order to deliver voltage and measure current flow; the resulting data carries information about the conductivity of the materials underground, allowing deductions about their composition (\cite{Uh09}). A similar method is also used in medical imaging. Since the tissues of a body have different electrical conductivities (\cite{Jo98}), using the same setup harmless currents can be allowed to flow in the body of a patient, thus collecting information about its internal structure. This technique can be applied to cancer detection (\cite{GZ03}), monitoring of vital functions (\cite{CGIN90}) and more (see e.g. \cite{Ho05}). Various engineering applications have also been proposed. A recent one (see \cite{HPS14}) describes a sensing skin consisting of a thin layer of conductive copper paint applied on concrete. In case of cracking of the block, the rupture of the surface would result in a local decrease in conductivity, which would in turn be detected by EIT, allowing the timely substitution of the failing block. The version of the problem with non-vanishing magnetic field is interesting on its own, as it is related to the inverse scattering problem with a fixed energy (see \cite{NSU95}). First order terms also arise by reduction in the study of numerous other inverse problems, among which isotropic elasticity (\cite{NU94}), special cases of Maxwell and Schr\"odinger equations (\cite{Mc00}, \cite{Es01}), Dirac equations (\cite{NT00}) and the Stokes system (\cite{HLW06}). The survey \cite{Sa07} contains more references on inverse boundary value problems for the MSE.

Below we introduce a fractional extension of the local problem. Fractional mathematical models are nowadays quite common in many different fields of science, including image processing (\cite{GO08}), physics (\cite{DGLZ2012}, \cite{Er02}, \cite{GL97}, \cite{La00}, \cite{MK00}, \cite{ZD10}), ecology (\cite{Hu10}, \cite{MV18}, \cite{RR09}), turbulent fluid dynamics (\cite{Co06}, \cite{DG13}) and mathematical finance (\cite{AB88}, \cite{Le04}, \cite{Sc03}). For more references, see \cite{BV18}. The common idea in these applications is that the fractional Schr\"odinger equation usefully describes anomalous diffusion, i.e. a diffusion process in which the mean squared displacement does not depend linearly on time. We expect this to be even more the case for FMSE, given its greater generality.

For the fractional case, fix $s\in(0,1)$, and consider the fractional divergence and gradient operators $(\nabla\cdot)^s$ and $\nabla^s$. These are based on the theoretical framework laid down in \cite{DGLZ2012}, \cite{DGLZ2013}, and were introduced in \cite{Co18} as non-local counterparts of the classical divergence and gradient. Fix a vector potential $A$, and consider the magnetic versions $(\nabla\cdot)^s_A$ and $\nabla^s_A$ of the above operators. These correspond to $(-i\nabla+A)\cdot$ and $(-i\nabla+A)$, whose combination results in the local magnetic Laplacian $(-\Delta)_A$. Analogously, we will show how $(\nabla\cdot)^s_A$ and $\nabla^s_A$ can be combined in a fractional magnetic Laplacian $(-\Delta)^s_A$.

\noindent The next step will be setting up the Dirichlet problem for FMSE as 
$$\left\{\begin{array}{lr}
        (-\Delta)^s_A u + qu =0 & \text{in } \Omega\\
        u=f & \text{in } \Omega_e
        \end{array}\right. \;.
				$$

Since our operators are non-local, the exterior values are taken over $\Omega_e=\mathbb R^n\setminus \overline\Omega$. The well-posedness of the direct problem is granted by the assumption that $0$ is not an eigenvalue for the left hand side of FMSE (see e.g. \cite{RS2017}). We can therefore define the DN map $\Lambda_{A,q}^s : H^s(\Omega_e) \rightarrow (H^{s}(\Omega_e))^*$ from the bilinear form associated to the equation. The inverse problem is to recover $A$ and $q$ in $\Omega$ from $\Lambda_{A,q}^s$. Because of a natural gauge $\sim$ enjoyed by FMSE, solving the inverse problem completely is impossible; however, the gauge class of the solving potentials can be fully recovered:

\begin{The}
Let $\Omega \subset \mathbb R^n, \; n\geq 2$ be a bounded open set, $s\in(0,1)$, and let $(A_i,q_i) \in \cal P$ for $i=1,2$. Suppose $W_1, W_2\subset \Omega_e$ are open sets, and that the DN maps for the FMSEs in $\Omega$ relative to $(A_1,q_1)$ and $(A_2,q_2)$ satisfy   $$\Lambda^s_{A_1,q_1}[f]|_{W_2}=\Lambda^s_{A_2,q_2}[f]|_{W_2}, \;\;\;\;\; \forall f\in C^\infty_c(W_1)\;.$$ \noindent Then $(A_1,q_1)\sim(A_2,q_2)$, that is, the potentials coincide up to the gauge $\sim$.  
\end{The}

The set $\mathcal P$ of potentials and the gauge $\sim$ are defined in Section 3. $\cal P$ contains all potentials $(A,q)$ satisfying certain properties, among which \emph{(p5)}: supp$(A)\subseteq\Omega^2$. We suspect this assumption to be unnecessary, but we nonetheless prove our Theorem in this easier case, and highlight the occasions when \emph{(p5)} is used.

The proof is based on three preliminary results: the integral identity for the DN map, the weak unique continuation property (WUCP) and the Runge approximation property (RAP). The WUCP is easily proved by reducing our case to that of the fractional Laplacian $(-\Delta)^s$, for which the result is already known (see e.g. \cite{Ru15}, \cite{GSU2017}). For this we use \emph{(p5)}. The proof of the RAP then comes from the WUCP and the Hahn-Banach theorem. Eventually, we use this result, the integral identity and \emph{(p5)} to complete the proof by means of Alessandrini's identity. This technique generalizes the one studied in \cite{GSU2017}. 

We consider Theorem 1.1 to be very satisfactory, as gauges show up in the local case $s=1$ as well (again, see \cite{NSU95}). For comparison see \cite{CLR18}, where it is shown that no gauge exists for a certain MSE in which only the highest order term is non-local. This interesting result inspired us to investigate a fully fractional operator, and is thus the main academic motivation for this work.  


\section{Preliminaries}

\para{Operators on bivariate vector functions} 

\begin{Def}\label{def1} Let $A \in C^{\infty}_c(\mathbb R^{n}\times\mathbb R^{n},\mathbb C^n)$. The \emph{symmetric, antisymmetric, parallel and perpendicular parts of $A$} at points $x,y$ are $$ A_s(x,y):= \frac{A(x,y)+A(y,x)}{2}\,, \;\;\;\;\;\; A_a(x,y) := A(x,y)-A_s(x,y)\;, $$
$$A_\parallel(x,y) := \left\{
    \begin{array}{cc}
       \frac{A(x,y)\cdot(x-y)}{|x-y|^2}(x-y) & \mbox{if } x\neq y \\
       A(x,y)       & \mbox{if } x=y
    \end{array} \right.\,,\;\; A_\perp(x,y) := A(x,y)-A_\parallel(x,y)\;.$$

\noindent The \emph{$L^2$ norms of $A$ with respect to the first} and \emph{second variable} at point $x$ are $$ {\cal J}_1 A(x) := \left( \int_{\mathbb R^n} |A(y,x)|^2 \,dy \right)^{1/2}\;\;, \;\;\;\; {\cal J}_2 A(x) := \left( \int_{\mathbb R^n} |A(x,y)|^2 \,dy \right)^{1/2}\;.$$
\end{Def}

\begin{Rem}
Being $A\in C^\infty_c$, these two integrals are finite and the definitions make sense. Moreover, since $A_a \cdot A_s$ is an antisymmetric scalar function and $A_\parallel \cdot A_\perp = 0$, by the following computations
\begin{equation} \label{pitsim}
\begin{split}   
\|A\|_{L^2}^2 & = \|A_a + A_s\|_{L^2}^2 = \|A_a\|_{L^2}^2+\|A_s\|_{L^2}^2+2\langle A_a, A_s\rangle \\ & = \|A_a\|_{L^2}^2+\|A_s\|_{L^2}^2+2\int_{\mathbb R^{2n}}A_a\cdot A_s \,dx\,dy = \|A_a\|_{L^2}^2+\|A_s\|_{L^2}^2\;\;,
\end{split}
\end{equation}
\begin{equation} \label{pitort}
\begin{split}
\|A\|_{L^2}^2 & = \|A_\parallel + A_\perp\|_{L^2}^2 = \|A_\parallel\|_{L^2}^2+\|A_\perp\|_{L^2}^2+2\langle A_\parallel, A_\perp\rangle \\ & = \|A_\parallel\|_{L^2}^2+\|A_\perp\|_{L^2}^2+2\int_{\mathbb R^{2n}}A_\parallel\cdot A_\perp \,dx\,dy = \|A_\parallel\|_{L^2}^2+\|A_\perp\|_{L^2}^2\;\;
\end{split}
\end{equation}

\noindent the four operators $(\cdot)_s, (\cdot)_a, (\cdot)_\parallel, (\cdot)_\perp$ can be extended to act from $L^2(\mathbb R^{2n})$ to $L^2(\mathbb R^{2n})$. This is true of ${\cal J}_1A$ and ${\cal J}_1A$ as well:
\spleq{normpart}{
\| {\cal J}_1A \|^2_{L^2(\mathbb R^n)} & = \int_{\mathbb R^n} |({\cal J}_1A)(x)|^2 \,dx = 
\int_{\mathbb R^{2n}} |A(y,x)|^2 \,dy \,dx = \|A\|^2_{L^2(\mathbb R^{2n})}\;. }
\end{Rem}

\begin{Lem}\label{almevlem} The equalities defining $(\cdot)_s, (\cdot)_a, (\cdot)_\parallel, (\cdot)_\perp$ in Definition \ref{def1} for $A\in C^\infty_c$ still hold a.e. for $A\in L^2(\mathbb R^{2n})$. 
\end{Lem}
\begin{proof}
We prove the Lemma only for $(\cdot)_s$, as the other cases are similar. For all $i\in\mathbb N$, let $A^{i}\in C^{\infty}_c(\mathbb R^{2n},\mathbb C^n)$ such that $\|A-A^{i}\|_{L^2} \leq 1/i$. By \eqref{pitsim},
\begin{align*}
\bigg\|A_s \bigg.&-\left.\frac{A(x,y)+A(y,x)}{2} \right\|_{L^2} \leq \\ & \leq  \|(A-A^{i})_s\|_{L^2} +\left\|A^{i}_s - \frac{A^{i}(x,y)+A^{i}(y,x)}{2} \right\|_{L^2} \\ & \;\;\;\; + \left\| \frac{(A(x,y) - A^{i}(x,y)) +(A(y,x) - A^{i}(y,x))}{2} \right\|_{L^2} \\ & = \|(A-A^{i})_s\|_{L^2} + \left\| \frac{(A(x,y) - A^{i}(x,y)) +(A(y,x) - A^{i}(y,x))}{2} \right\|_{L^2} \\ & \leq 2 \|A-A^{i}\|_{L^2} \leq 2/i\;.\qedhere\end{align*}\end{proof}

\begin{Rem}
If $A\in C^{\infty}_c$, the operators $(\cdot)_s, (\cdot)_a, (\cdot)_\parallel, (\cdot)_\perp$ commute with each other; because of Lemma \ref{almevlem}, this still holds a.e. for $A\in L^2(\mathbb R^{2n})$. Thus in the following we use e.g. the symbol $A_{s\parallel}$ for both $(A_s)_\parallel$ and $(A_\parallel)_s$.
\end{Rem}

\para{Sobolev spaces} Let $\Omega\subset \mathbb R^n$ be open and $r\in \mathbb R$, $p\in(1,\infty)$, $n\in\mathbb{N}\setminus \{0\}$. By the symbols $W^{r,p} = W^{r,p}(\mathbb R^n)$ and $W^{r,p}_c(\Omega)$ we denote the usual $L^p$-based Sobolev spaces. We also let $H^s = H^s(\mathbb{R}^n) = W^{s,2}(\mathbb{R}^n)$ be the standard $L^2$-based Sobolev space with norm $\|u\|_{H^s(\mathbb{R}^n)} = \|\mathcal{F}^{-1}( \langle\xi\rangle^s \hat u ) \|_{L^2(\mathbb{R}^n)}\;,$ where $s\in \mathbb R$, $\langle\xi\rangle := (1+|\xi|^2)^{1/2}$ and the Fourier transform is 
$$ \hat u(\xi) = \mathcal F u(\xi) = \int_{\mathbb R^n} e^{-ix\cdot\xi} u(x) dx\;.$$ 

One should note that there exist many equivalent definitions of fractional Sobolev spaces (see e.g. \cite{hitch}). Using the Sobolev embedding and multiplication theorems (see e.g. \cite{BBM01}, \cite{BH2017}), these spaces can often be embedded into each other: 

\begin{Lem} Let $s\in(0,1), \,p :=$\emph{max}$\{2, n/2s\}$ and $h\geq 0$. Then the embeddings 
\begin{multicols}{2}
\begin{enumerate}[label=(e\arabic*).]
\item $\; H^s \times H^s \hookrightarrow L^{n/(n/2+s p-2s)} \;,$
\item $\; H^s \times L^{p} \hookrightarrow L^{2n/(n+2s)}\;,$ 
\item $\; L^{2p} \times L^2 \hookrightarrow L^{2n/(n+2s)}\;,$  
\item $\; L^{2p} \times H^s \hookrightarrow L^2\;,$ 
\item $\; L^{2p} \times L^{2p} \hookrightarrow L^{p}\;,$  
\item $\; H^{sp-2s} \hookrightarrow L^{p}\;,$
\item $\; L^{2n/(n+2h)} \hookrightarrow H^{-h}$
\end{enumerate} 
\end{multicols}
\noindent hold, where $\times$ indicates the pointwise product.\qed\end{Lem}

Let $U, F\subset \mathbb R^n$ be an open and a closed set. We define the spaces $$H^s(U) = \{ u|_U, u\in H^s(\mathbb R^n) \}\;,$$$$ \tilde H^s(U) = \text{closure of $C^\infty_c(U)$ in $H^s(\mathbb R^n)$}\;, \;\mbox{and}$$ $$ H^s_F(\mathbb R^n) = \{ u\in H^s(\mathbb R^n) : \text{supp}(u) \subset F \}\;, $$

\noindent where $\|u\|_{H^s(U)}= \inf\{ \|w\|_{H^s(\mathbb R^n)} ; w\in H^s(\mathbb R^n), w|_{U}=u \}$. For $s\in(0,1)$ and a bounded open set $U\subset \mathbb R^n$, let $X:= H^s(\mathbb R^n)/\tilde H^s(U)$. If $U$ is a Lipschitz domain, then then $\tilde H^s(U)$ and $H^s_{\bar U}(\mathbb R^n)$ can be identified for all $s\in \mathbb R$ (see \cite{GSU2017}); therefore, $X = H^s(\mathbb R^n)/H^s_{\bar U}(\mathbb R^n)$, and its elements are equivalence classes of functions from $H^s(\mathbb R^n)$ coinciding on $U_e$. X is called \emph{abstract trace space}.

\para{Non-local operators} If $u\in \mathcal S(\mathbb R^n)$, its fractional Laplacian is (see \cite{Kw15}, \cite{hitch}) $$ (-\Delta)^s u(x) := \mathcal C_{n,s} \lim_{\epsilon \rightarrow 0^+}\int_{\mathbb R^n\setminus B_\epsilon (x)} \frac{u(x)-u(y)}{|y-x|^{n+2s}} dy\;,$$ 
\noindent for a constant $\mathcal C_{n,s}$. Its Fourier symbol is $|\xi|^{2s}$, i.e. $(-\Delta)^s u(x) =\mathcal F^{-1} ( |\xi|^{2s} \hat u (\xi) )$. By \cite{Ho90}, Ch. 4 and \cite{Ta96}, $(-\Delta)^s$ extends as a bounded map $(-\Delta)^s : W^{r,p}(\mathbb R^n) \rightarrow W^{r-2s,p}(\mathbb R^n)$ for $r\in\mathbb R$ and $p\in(1,\infty)$. Let $\alpha(x,y): \mathbb R^{2n}\rightarrow \mathbb R^n$ be the map 
\begin{equation*}\alpha(x,y)= \frac{\mathcal C_{n,s}^{1/2}}{\sqrt 2} \frac{y-x}{|y-x|^{n/2 +s+1}}\;.\end{equation*} 

\noindent If $u\in C^{\infty}_c(\mathbb R^n)$ and $x,y \in \mathbb R^n$, the \emph{fractional gradient of $u$ at points $x$ and $y$} is \begin{equation}\label{graddef}\nabla^s u(x,y) := (u(x)-u(y))\alpha(x,y)\;,\end{equation}

\noindent and is thus a symmetric and parallel vector function of $x$ and $y$. Since it was proved in \cite{Co18} that $\|\nabla^s u\|^2_{L^2(\mathbb R^{2n})} \leq \|u\|^2_{H^s(\mathbb R^n)}$, and thus that the linear operator $\nabla^s$ maps $C^\infty_c(\mathbb R^n)$ into $L^2(\mathbb R^{2n})$, we see that $\nabla^s$ can be extended to $\nabla^s : H^s(\mathbb R^n) \rightarrow L^2(\mathbb R^{2n})$. Using a proof by density similar to the one for Lemma \ref{almevlem}, one sees that \eqref{graddef} still holds a.e. for $u\in H^s(\mathbb R^n)$.

\noindent If $u\in H^s(\mathbb R^n)$ and $v\in L^2(\mathbb R^{2n})$, the \emph{fractional divergence} is defined as that operator $(\nabla\cdot)^s : L^2(\mathbb R^{2n}) \rightarrow H^{-s}(\mathbb R^n)$ satisfying
\begin{equation}
\label{eq:defdiv}
\langle (\nabla\cdot)^s v,u \rangle_{L^2(\mathbb R^{n})} = \langle v,\nabla^s u \rangle_{L^2(\mathbb R^{2n})}  \;,
\end{equation}

\noindent i.e. it is by definition the adjoint of the fractional gradient. As observed in \cite{Co18}, Lemma 2.1, if $u \in H^s(\mathbb R^n)$ the equality $(\nabla\cdot)^s(\nabla^su)(x) = (-\Delta)^su(x)$ holds in weak sense, and $(\nabla\cdot)^s(\nabla^su) \in H^{-s}(\mathbb R^n)$. 

\begin{Lem} Let $u\in C^\infty_c(\mathbb R^n)$. There exists a constant $k_{n,s}$ such that $${\cal F}(\nabla^s u)(\xi,\eta) = k_{n,s} \left( \frac{\xi}{|\xi|^{n/2+1-s}}+\frac{\eta}{|\eta|^{n/2+1-s}} \right){\cal F} u(\xi+\eta)\;.$$
\end{Lem}
\begin{proof} As $u\in C^\infty_c(\mathbb R^n)$, we know that $\nabla^s u \in L^2(\mathbb R^{2n})$, and we can compute its Fourier transform in the variables $\xi,\eta$. By a change of variables,
\begin{equation*}\begin{split}
{\cal F}(\nabla^s u)(\xi,\eta) & = \frac{\mathcal C_{n,s}^{1/2}}{\sqrt 2}\int_{\mathbb R^n}\int_{\mathbb R^n} e^{-ix\cdot\xi}e^{-iy\cdot\eta} \frac{u(x)-u(y)}{|y-x|^{n/2 +s+1}}(y-x) \; dx\,dy \\ & = k'_{n,s}\int_{\mathbb R^n} \frac{e^{-iz\cdot\eta}}{|z|^{n/2 +s+1}} z \int_{\mathbb R^n} e^{-ix\cdot(\xi+\eta)} (u(x)-u(x+z)) \; dx\,dz \\ & = k'_{n,s}\int_{\mathbb R^n} \frac{z}{|z|^{n/2 +s+1}} e^{-iz\cdot\eta} \,{\cal F}u(\xi+\eta) (1-e^{iz\cdot(\xi+\eta)})\,dz \\ & = k''_{n,s}\, {\cal F}u(\xi+\eta) \int_{\mathbb R^n} (e^{-iz\cdot\eta} - e^{iz\cdot\xi})\nabla_z(|z|^{1-n/2 -s}) \,dz \\ & = k''_{n,s}\, {\cal F}u(\xi+\eta) \left(\eta{\cal F}(|z|^{1-n/2 -s})(\eta)+\xi{\cal F}(|z|^{1-n/2 -s})(-\xi)\right) \\ & = k_{n,s} \left( \frac{\xi}{|\xi|^{n/2+1-s}}+\frac{\eta}{|\eta|^{n/2+1-s}} \right){\cal F} u(\xi+\eta)\;.\qedhere
\end{split}\end{equation*}\end{proof}

\begin{Lem}\label{extgrad}
The fractional gradient extends as a bounded map $$\nabla^s : H^r(\mathbb R^n) \rightarrow \langle D_x+D_y \rangle^{r-s} L^2(\mathbb R^{2n})\;,$$
\noindent and if $r\leq s$ then also $\; \nabla^s : H^r(\mathbb R^n) \rightarrow H^{r-s}(\mathbb R^{2n})\;. $
\end{Lem}
\begin{proof}

\noindent Start with $u\in C^\infty_c(\mathbb R^n)$, and let $r\in\mathbb R$. Then
\spleq{estigrad1}{\|\nabla^s u\|^2_{\langle D_x+D_y \rangle^{r-s} L^2} & = ( \langle D_x+D_y \rangle^{r-s} \nabla^s u,  \langle D_x+D_y \rangle^{r-s} \nabla^s u )_{L^2} \\ & = ( \langle D_x+D_y \rangle^{2(r-s)} \nabla^s u, \nabla^s u )_{L^2} \\ & = ({\cal F}( \langle D_x+D_y \rangle^{2(r-s)} \nabla^s u), {\cal F}(\nabla^s u) )_{L^2}\;.
}
\noindent From the previous Lemma we can deduce that
\spleqno{ {\cal F}( \langle D_x+D_y \rangle&^{2(r-s)} \nabla^s u)  = (1+|\xi+\eta|^2)^{r-s} {\cal F}(\nabla^s u) \\ & = (1+|\xi+\eta|^2)^{r-s}  k_{n,s} \left( \frac{\xi}{|\xi|^{n/2+1-s}}+\frac{\eta}{|\eta|^{n/2+1-s}} \right){\cal F} u(\xi+\eta) \\ & = k_{n,s}  \left( \frac{\xi}{|\xi|^{n/2+1-s}}+\frac{\eta}{|\eta|^{n/2+1-s}} \right) {\cal F}( \langle D_x \rangle^{2(r-s)} u )(\xi+\eta) \\ & = {\cal F} (\nabla^s ( \langle D_x\rangle^{2(r-s)} u ))\;.}
Using the properties of the fractional gradient and \eqref{estigrad1},
\spleqno{  \|\nabla^s u\|^2_{\langle D_x+D_y \rangle^{r-s} L^2} & = ( {\cal F} (\nabla^s ( \langle D_x\rangle^{2(r-s)} u )), {\cal F}(\nabla^s u) )_{L^2} \\ & = ( \nabla^s ( \langle D_x\rangle^{2(r-s)} u ), \nabla^s u )_{L^2} = ( \langle D_x\rangle^{2(r-s)} u, (-\Delta)^s u )_{L^2} \\ & = (  \langle D_x\rangle^{r-s} (-\Delta)^{s/2} u, \langle D_x\rangle^{r-s} (-\Delta)^{s/2} u )_{L^2} \\ & = \| (-\Delta)^{s/2} u \|^2_{H^{r-s}} \leq c \|u\|^2_{H^r}\;.
}

\noindent An argument by density completes the proof of the first part of the statement. For the second one, observe that $r \leq s$ implies 
\spleqno{ \|v\|^2_{H^{r-s}} & = ( \langle D_{x,y}\rangle^{r-s} v, \langle D_{x,y}\rangle^{r-s} v )_{L^2} = ( \langle D_{x,y}\rangle^{2(r-s)} v, v )_{L^2} \\ & = ( (1+ |\xi|^2 + |\eta|^2)^{r-s} \hat v, \hat v )_{L^2} \leq  c( (1+ |\xi+\eta|^2)^{r-s} \hat v, \hat v )_{L^2} \\ & =c ( \langle D_x + D_y \rangle^{2(r-s)}v,v )_{L^2} = c \|v\|^2_{\langle D_x+D_y \rangle^{r-s} L^2}\;,}
\noindent and so $ \langle D_x+D_y \rangle^{r-s} L^2(\mathbb R^{2n}) \subseteq H^{r-s}(\mathbb R^{2n})$.\end{proof}

\noindent As a consequence of the above Lemma, the fractional divergence can be similarly extended as $(\nabla\cdot)^s : H^t(\mathbb R^{2n})\rightarrow H^{t-s}(\mathbb R^n)$ for all $t\geq s$.

\section{Definition and properties of FMSE}


\para{Fractional magnetic Schr\"odinger equation} Let $\Omega\subset \mathbb R^n$ be open, $\Omega_e = \mathbb R^n\setminus\overline\Omega$ be the \emph{exterior domain}, and also recall that $p :=$max$\{2, n/2s\}$. The \emph{vector potential} and \emph{scalar potential} are two functions $A: \mathbb R^{2n}\mapsto \mathbb C^n$ and $q:\mathbb R^n\mapsto\mathbb R$. The following properties are of interest:

\begin{enumerate}[label=\emph{(p\arabic*)}.]
\item $\;{\cal J}_1A, \;{\cal J}_2A \in L^{2p}(\mathbb R^n)\;,$
\item $\;A_{s\parallel} \in H^{sp-s}(\mathbb R^{2n}, \mathbb C^n)\;,$
\item $\;A_{a\parallel}(x,y) \cdot (y-x) \geq 0, \;$ for all $x, y\in \mathbb R^n\;,$
\item $\;q\in L^{p}(\Omega) \;,$
\item $\;A\in L^2(\mathbb R^{2n}), \;\;\;$ supp$(A) \subseteq \Omega^2\;.$
\end{enumerate}

\noindent With respect to the above properties, we define four sets of potentials: 
$$ \begin{array}{c}
{\cal A}_0 := \{ \mbox{vector potentials } A  \mbox{ verifying } (p1) - (p3) \}, \\
{\cal A} := \{ \mbox{vector potentials } A  \mbox{ verifying } (p1) - (p3) \mbox{ and } (p5) \}, \\
{\cal P}_0 := \{ \mbox{pairs of potentials } (A,q)  \mbox{ verifying } (p1) - (p4) \}, \\
{\cal P} := \{ \mbox{pairs of potentials } (A,q)  \mbox{ verifying } (p1) - (p5) \}. \\

\end{array} $$

\begin{Rem}
The peculiar definitions for the spaces in \emph{(p1)}, \emph{(p2)} and \emph{(p4)} are due to computational necessities: they make the following quantities
$$\|qu\|_{H^{-s}} ,\;\;\; \|(\nabla\cdot)^sA_{s\parallel}\|_{L^{p}} ,\;\;\; \|(\mathcal J_2A)^2\|_{L^{p}},\;\;\; \|u \mathcal J_2A \|_{L^2} $$

\noindent finite for $u\in H^s$, as needed in Remark 3.8, Lemma 3.12 and \eqref{usingp1}. This is easily proved by using Lemma 2.5. However, if $n\geq 4$, then $p=n/2s$, and so in this case $L^{2p} = L^{n/s}$ and $H^{sp-s}=H^{n/2-s}$; this simplifies the assumptions for $n$ large enough.   

\end{Rem}

\noindent Let $A \in {\cal A}_0$ and $u\in H^s(\mathbb R^n)$. By \emph{(p1)} and \emph{(e4)},
\spleq{usingp1}{
\| A(x,y)u(x) \|_{L^2(\mathbb R^{2n})} & = \left(\int_{\mathbb R^n} u(x)^2 \int_{\mathbb R^n} |A(x,y)|^2 dy\, dx\right)^{1/2} \\ & = \left(\int_{\mathbb R^n} u(x)^2 \,{\cal J}_2A(x)^2 \,dx\right)^{1/2} = \|u \,{\cal J}_2A\|_{L^2(\mathbb R^n)} \\ & \leq k \|u\|_{H^s} \|{\cal J}_2A\|_{L^{2p}}<\infty\;, 
}
\noindent and thus the \emph{magnetic fractional gradient} of $u$ can be defined as the function $\nabla^s_A u : \mathbb R^{2n}\rightarrow \mathbb C^n$ such that \begin{equation}\label{defgradA} \langle \nabla^s_A u, v \rangle := \langle \nabla^s u + A(x,y)u(x), v \rangle\;, \;\;\; \mathrm{for}\;\mathrm{all}\; v\in L^2(\mathbb R^{2n})\;. \end{equation}

\noindent By the same computation, $\nabla^s_A$ acts as an operator $\nabla^s_A : H^s(\mathbb R^n) \rightarrow L^2(\mathbb R^{2n})$. 

\noindent Let $A\in {\cal A}_0$, $u\in H^s({\mathbb R^n})$ and $v\in L^2(\mathbb R^{2n})$. The \emph{magnetic fractional divergence} is defined by duality as that operator $(\nabla\cdot)^s_A : L^2(\mathbb R^{2n}) \rightarrow H^{-s}({\mathbb R^n})$ such that
\spleqno{
\langle (\nabla\cdot)^s_A v, u \rangle := \langle v, \nabla^s_A u \rangle\;.
}

\noindent By construction, the magnetic fractional divergence and gradient can be combined; we call \emph{magnetic fractional Laplacian} $(-\Delta)^s_A:=(\nabla\cdot)^s_A(\nabla^s_A)$ that operator from $H^s(\mathbb R^n)$ to $H^{-s}(\mathbb R^n)$ such that, for all $u,v \in H^s(\mathbb R^n)$,  
\spleq{deflapA}{
\langle (-\Delta)^s_A u, v \rangle = \langle \nabla^s_A u, \nabla^s_A v \rangle\;.}

\begin{Rem}\label{Anonzero}
If $A\equiv 0$, the magnetic fractional Laplacian $(-\Delta)^s_A$ is reduced to its non-magnetic counterpart $(-\Delta)^s$, as expected. Since the fractional Laplacian is well understood (see e.g. \cite{GSU2017}), from now on we assume $A\not\equiv 0$.  
\end{Rem}

\begin{Lem}\label{lemform1} Let $A\in L^2(\mathbb R^{2n}) \cap {\cal A}_0$ and $u\in H^s(\mathbb R^n)$. The equation \spleq{form1}{(-\Delta)^s_A u = (-\Delta)^su + 2\int_{\mathbb R^n} \left(A_{a\parallel} \cdot \nabla^s u\right) \,dy + \left( (\nabla\cdot)^sA_{s\parallel} + \int_{\mathbb R^n} |A|^2 \, dy \right) u}
\noindent holds in weak sense.
\end{Lem}
\begin{proof} 
\noindent By \eqref{deflapA}, $(-\Delta)^s_A u \in H^{-s}(\mathbb R^n)$, and in order to prove \eqref{form1} in weak sense one needs to compute $\langle (-\Delta)^s_A u, v \rangle$ for $v\in H^s(\mathbb R^n)$. By \eqref{deflapA} and \eqref{defgradA}, 
\spleqno{ \langle (-\Delta)^s_A u, v \rangle & = \langle \nabla^s u + A(x,y)u(x), \nabla^s v + A(x,y)v(x) \rangle \\ & = \langle\nabla^s u,\nabla^s v\rangle + \langle Au, Av\rangle + \langle\nabla^s u, Av\rangle + \langle \nabla^s v, Au\rangle\;, }
\noindent where all the above terms make sense, since by formula \eqref{usingp1} $\nabla^s u, \nabla^s v, Au$ and $Av$ all belong to $L^2(\mathbb R^{2n})$. The new term $\langle \nabla^s u, A(y,x)v(x) \rangle$ is also finite, so
\spleq{breaklap}{  \langle (-\Delta)^s_A u, v \rangle  = \,& \langle\nabla^s u,\nabla^s v\rangle + \langle Au, Av\rangle + \\ & + \langle\nabla^s u, A(x,y)v(x)\rangle -\langle \nabla^s u, A(y,x)v(x)\rangle + \\ & + \langle \nabla^s u, A(y,x)v(x) \rangle + \langle \nabla^s v, A(x,y)u\rangle\;. }
\noindent For the first term on the right hand side of \eqref{breaklap}, by definition, 
\spleq{firstterm}{
\langle\nabla^s u,\nabla^s v\rangle = \langle(\nabla\cdot)^s\nabla^s u, v\rangle = \langle (-\Delta)^s u, v \rangle\;.
} 
\noindent For the second one, by the embeddings \emph{(e5)}, \emph{(e2)} and $\emph{(e7)}$,
\spleq{secondterm}{
\langle Au, Av\rangle = \left\langle u(x) \int_{\mathbb R^n} |A(x,y)|^2 dy, v\right\rangle = \langle u ({\cal J}_2A)^2, v\rangle \;.
}
\noindent Since $u\in H^s(\mathbb R^n)$, by \eqref{normpart} we deduce ${\cal J}_2(\nabla^s u) \in L^2(\mathbb R^n)$. Now \emph{(e3)} implies that ${\cal J}_2(\nabla^s u){\cal J}_2A \in L^{\frac{2n}{n+2s}}$. On the other hand, by Cauchy-Schwarz
\spleqno{ 
\left\| \int_{\mathbb R^n} \right.&\left.\nabla^s u \cdot A dy \right\|^{\frac{2n}{n+2s}}_{L^{\frac{2n}{n+2s}}(\mathbb R^n)} = \int_{\mathbb R^n} \left| \int_{\mathbb R^n} \nabla^s u \cdot A dy \right|^{\frac{2n}{n+2s}} dx \\ & \leq \int_{\mathbb R^n} \left( \int_{\mathbb R^n} |\nabla^s u|\,|A| dy \right)^{\frac{2n}{n+2s}} dx \leq \int_{\mathbb R^n} \left( \int_{\mathbb R^n} |\nabla^s u|^2 dy \int_{\mathbb R^n} |A|^2 dy\right)^{\frac{n}{n+2s}} dx \\ & = \int_{\mathbb R^n} | {\cal J}_2(\nabla^s u) \; {\cal J}_2A |^{\frac{2n}{n+2s}} dx = \left\| {\cal J}_2(\nabla^s u) \; {\cal J}_2A \right\|^{\frac{2n}{n+2s}}_{L^{\frac{2n}{n+2s}}(\mathbb R^n)}\;,
 }
\noindent and so $\int_{\mathbb R^n} \nabla^s u \cdot A dy \in L^{\frac{2n}{n+2s}}$. Now $\langle \int_{\mathbb R^n} \nabla^s u \cdot A dy  , v\rangle$ is finite by \emph{(e7)}, and
\spleq{thirdfourthterm}{
\left\langle\nabla^s u, A(x,y)\right.&\left.v(x)\right\rangle -\langle \nabla^s u, A(y,x)v(x)\rangle = \\ & = \left\langle \int_{\mathbb R^n} \nabla^s u \cdot A(x,y) dy  , v\right\rangle - \left\langle \int_{\mathbb R^n} \nabla^s u \cdot A(y,x) dy  , v\right\rangle \\ & = \left\langle \int_{\mathbb R^n} \nabla^s u \cdot (A(x,y)-A(y,x)) dy  , v\right\rangle \\ & = \left\langle 2\int_{\mathbb R^n} \nabla^s u \cdot A_a\, dy  , v\right\rangle = \left\langle 2\int_{\mathbb R^n} \nabla^s u \cdot A_{a\parallel}\, dy  , v\right\rangle\;.
}
\noindent The last steps use Lemma \ref{almevlem} to write $A_a$ for $A\in L^2$ and to see that $\nabla^s u$ is a.e. a parallel vector for $u\in H^s(\mathbb R^n)$, which implies $\nabla^s u \cdot A_{a\perp} = 0 \;a.e.$. This computes the third and fourth terms on the right hand side of \eqref{breaklap}. For the last two terms observe that, since $A(y,x)v(x) - A(x,y)v(y)$ is antisymmetric, by Lemma \ref{almevlem} we have $ \langle \nabla^s u , A(y,x)v(x) - A(x,y)v(y) \rangle =0$, and so 
\spleq{lastterm}{ \left\langle \nabla^s u, A(y,x)\right.&\left.v(x) \right\rangle + \langle \nabla^s v, Au\rangle\ \\ & = \int_{\mathbb R^{2n}} A(x,y)\cdot (v(y)\nabla^s u + u(x)\nabla^s v )\; dx\,dy\; \\ & =
\int_{\mathbb R^{2n}} A\cdot \alpha \Big(v(y)(u(x)-u(y)) + u(x)(v(x)-v(y)) \Big)\; dx\,dy \\ & = \int_{\mathbb R^{2n}} A_{s\parallel}\cdot \alpha \Big(u(x)v(x) - u(y)v(y) \Big)\; dx\,dy \\ & = \langle A_{s\parallel}, \nabla^s(uv)\rangle = \langle u (\nabla\cdot)^s A_{s\parallel}, v \rangle \;.  
}
\noindent On the third line of \eqref{lastterm} the integrand is the product of a symmetric, parallel vector and $A$; this reduces $A$ to $A_{s\parallel}$. From \emph{(e1)}, \emph{(e7)} and Lemma \ref{extgrad} one sees that $\nabla^s(uv) \in H^{s-sp}$, and now $\langle A_{s\parallel}, \nabla^s(uv)\rangle$ makes sense by \emph{(p2)}. Eventually, (\ref{eq:defdiv}), \emph{(e6)}, \emph{(e2)} and \emph{(e7)} explain the last step. Equation \eqref{form1} follows from \eqref{breaklap}, \eqref{firstterm}, \eqref{secondterm}, \eqref{thirdfourthterm} and \eqref{lastterm}. \end{proof}

\begin{Lem}\label{lemsigma} Let $A\in L^2(\mathbb R^{2n}) \cap {\cal A}_0$. There exists a positive, symmetric distribution $\sigma\in {\cal D}'(\mathbb R^{2n})$ such that $A_{a\parallel} = \alpha(\sigma-1)$ a.e..
\end{Lem}
\begin{proof} Because of Lemma \ref{almevlem}, $A_{a\parallel}$ is a parallel vector almost everywhere, and thus $\|A_{a\parallel} - (A_{a\parallel})_\parallel\|_{L^2} =0$. Again by Lemma \ref{almevlem},
\spleqno{
0 & = \|A_{a\parallel} - (A_{a\parallel})_\parallel\|_{L^2}  = \left\| A_{a\parallel} - \frac{A_{a\parallel} \cdot (x-y)}{|x-y|^2}(x-y) \right\|_{L^2} \\ & = \left\| A_{a\parallel} - \left( -\frac{\sqrt 2}{\mathcal C_{n,s}^{1/2}} \frac{A_{a\parallel} \cdot (x-y)}{|x-y|^{1-n/2-s}}\right) \frac{\mathcal C_{n,s}^{1/2}}{\sqrt 2} \frac{y-x}{|y-x|^{n/2 +s+1}} \right\|_{L^2} \\ & = \left\| A_{a\parallel} - \left(\left(1+ \frac{\sqrt 2}{\mathcal C_{n,s}^{1/2}} \frac{A_{a\parallel} \cdot (y-x)}{|x-y|^{1-n/2-s}}\right)-1\right) \alpha \right\|_{L^2}\;.
}
\noindent Moreover, if $\phi\in C^\infty_c(\mathbb R^{2n})$ and $B_{r_1}, B_{r_2}$ are balls in $\mathbb R^n$ centered at the origin such that supp$(\phi)\subset B_{r_1} \times B_{r_2}$, then by \eqref{pitsim}, \eqref{pitort} and Cauchy-Schwarz inequality 
\spleqno{
\left|\left\langle 1+\frac{\sqrt{2}}{C_{n,s}^{1/2}}\right.\right.&\left.\left. |y-x|^{n/2+s} \left( A_{a\parallel} \cdot \frac{y-x}{|y-x|} \right) ,\phi \right\rangle\right| = \\ & = \left| \int_{\mathbb R^n}\int_{\mathbb R^n} \left(1+\frac{\sqrt{2}}{C_{n,s}^{1/2}}|y-x|^{n/2+s} \left( A_{a\parallel} \cdot \frac{y-x}{|y-x|} \right)\right)\phi\;dy\,dx \right| \\ & \leq  \int_{\mathbb R^n}\int_{\mathbb R^n} \left|1+\frac{\sqrt{2}}{C_{n,s}^{1/2}}|y-x|^{n/2+s} \left( A_{a\parallel} \cdot \frac{y-x}{|y-x|} \right)\right||\phi|\;dy\,dx \\ & \leq \|\phi\|_{L^\infty} \int_{B_{r_1}}\int_{B_{r_2}} \left|1+\frac{\sqrt{2}}{C_{n,s}^{1/2}}|y-x|^{n/2+s} \left( A_{a\parallel} \cdot \frac{y-x}{|y-x|} \right)\right|\;dy\,dx \\ & \leq k \|\phi\|_{L^\infty} \left( 1 + \int_{B_{r_1}}\int_{B_{r_2}} |y-x|^{n/2+s} \left| A_{a\parallel} \cdot \frac{y-x}{|y-x|} \right|\;dy\,dx \right) \\ & \leq k \|\phi\|_{L^\infty} \left( 1 + \int_{B_{r_1}}\int_{B_{r_2}} (|x|+|y|)^{n/2+s} | A_{a\parallel}|\;dy\,dx \right) \\ & \leq k'\|\phi\|_{L^\infty} \left( 1 + \int_{B_{r_1}}\int_{B_{r_2}} | A_{a\parallel}|\;dy\,dx \right) \\ & \leq  k'\|\phi\|_{L^\infty} \left( 1 + \|A_{a\parallel}\|^2_{L^2(\mathbb R^{2n})} \right) \leq  k'\|\phi\|_{L^\infty} \left( 1 + \|A\|^2_{L^2(\mathbb R^{2n})} \right) < \infty\;.
}

\noindent Thus it makes sense to define a distribution $\sigma\in {\cal D}'(\mathbb R^{2n})$ such that $$ \langle\sigma,\phi\rangle= \left\langle 1+\frac{\sqrt{2}}{C_{n,s}^{1/2}} |y-x|^{n/2+s} \left( A_{a\parallel} \cdot \frac{y-x}{|y-x|} \right) ,\phi \right\rangle $$
\noindent holds for all $\phi\in C^\infty_c(\mathbb R^{2n})$. Given that $A_{a\parallel}$ is antisymmetric, it is clear that $\sigma$ is symmetric; moreover, property \emph{(p3)} assures that $\sigma \geq 1$. \end{proof}

\begin{Rem}\label{lapcond}
If $u\in {\cal S}(\mathbb R^n)$, by the previous Lemma we can rewrite the leading term of $(-\Delta)^s_A$ as
\spleqno{
\mathcal C_{n,s}\; PV \int_{\mathbb R^n} \sigma(x,y) \frac{u(x)-u(y)}{|x-y|^{n+2s}}\, dy\;.
}
\noindent This shows the connection between the magnetic and classical fractional Laplacians: if $\sigma(x,y) \equiv 1$, i.e. if $A_{a\parallel}\equiv 0$, the formula above defines $(-\Delta)^su$. Moreover, if $\sigma(x,y)$ is \emph{separable} (i.e. there are functions $\sigma_1, \sigma_2 : \mathbb R^n \rightarrow \mathbb R$ such that $\sigma(x,y)=\sigma_1(x)\sigma_2(y)$) we get the fractional conductivity operator (see \cite{Co18}).
\end{Rem}

Consider $(A,q)\in {\cal P}_0$ and $f\in H^s(\Omega_e)$. We say that $u\in H^s(\mathbb R^n)$ solves FMSE with exterior value $f$ if and only if
$$ \left\{\begin{array}{lr}
        (-\Delta)^s_A u + qu = 0 & \text{in } \Omega\,\;\\
        u=f & \text{in } \Omega_e
        \end{array}\right.$$

\noindent holds in weak sense, that is if and only if $u-f \in \tilde H^s(\Omega)$ and, for all $v\in H^s(\mathbb R^n)$,\spleq{FMSE}{
\langle(-\Delta)^s_A u, v\rangle + \langle qu, v \rangle = 0\;.
}

\begin{Rem}
By \emph{(p1), (p2)} and \emph{(p4)}, formula \eqref{FMSE} makes sense. This was already partially shown in the above discussion about the magnetic fractional Laplacian. For the last term, just use \emph{(p4)}, \emph{(e2)} and \emph{(e7)}. 
\end{Rem}


\para{Old gauges, new gauges} Let $(G,\cdot)$ be the abelian group of all strictly positive functions $\phi\in C^\infty(\mathbb R^n)$ such that $\phi|_{\Omega_e} = 1$. For $(A,q), (A',q') \in {\cal P}_0$, define
\begin{gather}
(A,q) \sim (A',q') \;\;\;\Leftrightarrow\;\;\; (-\Delta)^s_A u + qu = (-\Delta)^s_{A'} u + q'u\;, \label{newg} \\
(A,q) \approx (A',q') \;\;\;\Leftrightarrow\;\;\; \exists \phi \in G : (-\Delta)^s_A (u\phi) + qu\phi = \phi( (-\Delta)^s_{A'} u + q'u) \label{oldg}
\end{gather}  

\noindent for all $u\in H^s(\mathbb R^n)$. Both $\sim$ and $\approx$ are equivalence relations on ${\cal P}_0$, and thus we can consider the quotient spaces ${\cal P}_0/\sim$ and ${\cal P}_0/\approx$. Moreover, since $\phi \equiv 1 \in G$, we have $(A,q) \sim (A',q') \Rightarrow (A,q) \approx (A',q')$.

We say that FMSE \emph{has the gauge} $\sim$ if for each $(A,q)\in {\cal P}_0$ there exists $(A',q')\in {\cal P}_0$ such that $(A',q')\neq(A,q)$ and $(A,q)\sim(A',q')$. Similarly, we say that FMSE \emph{has the gauge} $\approx$ if for each $(A,q)\in {\cal P}_0$ there exist $\phi\in G$, $(A',q')\in {\cal P}_0$ such that $\phi\not\equiv 1$, $(A',q')\neq(A,q)$ and $(A,q)\approx(A',q')$.

\begin{Rem}
The definitions \eqref{newg} and \eqref{oldg}, which have been given for FMSE, can be extended to the local case in the natural way. 
\end{Rem}

If $s=1$, it is known that $(-\Delta)_A (u\phi) + qu\phi = \phi \left( (-\Delta)_{A+\frac{\nabla\phi}{\phi}} u + qu \right)$ for all $\phi\in G$ and $u\in H^1(\mathbb R^n)$. If we choose $\phi\not\equiv 1$, we have $\left(A+\frac{\nabla\phi}{\phi},q\right)\neq(A,q)$ and $(A,q)\approx\left(A+\frac{\nabla\phi}{\phi},q\right)$, which shows that MSE has the gauge $\approx$. On the other hand, if $(A,q)\sim(A',q')$ then necessarily $A=A'$ and $q=q'$: thus, MSE does not enjoy the gauge $\sim$. We now treat the case $s\in(0,1)$.

\begin{Lem}\label{charofnewg}
Let $(A,q), (A',q') \in {\cal P}_0$. Then $(A,q)\sim(A',q')$ if and only if $A_{a\parallel} = A'_{a\parallel}$ and $Q = Q'$, where $$ Q := q + \int_{\mathbb R^n}|A|^2\, dy + (\nabla\cdot)^sA_{s\parallel}\;,\;\;\;\;\;\; Q' := q' + \int_{\mathbb R^n}|A'|^2\, dy + (\nabla\cdot)^sA'_{s\parallel}\;. $$
\end{Lem}
\begin{proof}
One direction of the implication is trivial: by \eqref{form1} and the definition, it is clear that if $A_{a\parallel} = A'_{a\parallel}$ and $Q=Q'$ then $ (-\Delta)^s_A u + qu = (-\Delta)^s_{A'} u + q'u $. 

\noindent For the other one, use Lemma 3.2 to write $ (-\Delta)^s_A u + qu = (-\Delta)^s_{A'} u + q'u $ as  
\spleq{comput1}{
0 &= 2\int_{\mathbb R^n} |\alpha|^2 (\sigma' - \sigma)(u(y)-u(x))\,dy + u(x) (Q-Q') \\ & = \mathcal C_{n,s}\int_{\mathbb R^n} (\sigma' - \sigma)\frac{u(y)-u(x)}{|x-y|^{n+2s}}\,dy + u(x) (Q-Q') \;. 
}

\noindent Fix $\psi \in C^{\infty}_c(\mathbb R^n)$, $x\in \mathbb R^n$ and $u(y):= \psi(y)e^{-1/|x-y|} |x-y|^{n+2s}$; one sees that $u\in \cal S$, since it is compactly supported and all the derivatives of the smooth function $e^{-1/|x-y|}$ vanish at $x$. Thus $u\in H^s$, and we can substitute it in \eqref{comput1}: 
$$0= \int_{\mathbb R^n} (\sigma(x,y) - \sigma'(x,y))e^{-1/|x-y|}\psi(y) \,dy  = \langle (\sigma(x,\cdot) - \sigma'(x,\cdot))e^{-1/|x-y|} ,\psi\rangle\;.$$
\noindent Being $\psi$ arbitrary and $e^{-1/|x-y|}$ non-negative, we deduce that $y\mapsto\sigma(x,y) - \sigma'(x,y)$ is zero for any fixed $x$, that is, $\sigma = \sigma'$. Then $A_{a\parallel} = A'_{a\parallel}$ by Lemma \ref{lemsigma}, and also $Q=Q'$ by \eqref{comput1}. \end{proof}

\begin{Lem}
Let $A\not\equiv 0$. Then FMSE has the gauge $\sim$. 
\end{Lem}
\begin{proof}
If $(A,q) \in {\cal P}_0$ and $A'\in {\cal A}_0$ is such that $A_{a\parallel} = A'_{a\parallel}$, then by the previous Lemma $(A,q)\sim(A',q')$ if and only if $Q=Q'$, that is
$$ q' = q + \int_{\mathbb R^n}|A|^2\, dy + (\nabla\cdot)^sA_{s\parallel} - \int_{\mathbb R^n}|A'|^2\, dy - (\nabla\cdot)^sA'_{s\parallel}\;.$$

Since $A, A' \in {\cal A}_0$, we have $A_{s\parallel}, A'_{s\parallel} \in H^{sp-s}$ and $\mathcal J_2A, \mathcal J_2A' \in L^{2p}$. By the former fact, $(\nabla\cdot)^sA_{s\parallel}, (\nabla\cdot)^sA'_{s\parallel}$ belong to $H^{sp-2s}$ and eventually to $L^{p}$ because of \emph{(e6)}. By the latter fact and \emph{(e5)}, $\int_{\mathbb R^n}|A|^2\, dy, \int_{\mathbb R^n}|A'|^2\, dy \in L^{p}$. Also, $q\in L^{p}$ because $(A,q)\in{\cal P}_0$. This implies that \emph{(p4)} holds for the $q'$ computed above. Hence, if we find $A'\in {\cal A}_0$ such that $A_{a\parallel} = A'_{a\parallel}$, and then take $q'$ as above, we get a $(A',q') \in {\cal P}_0$ in gauge $\sim$ with a given $(A,q) \in {\cal P}_0$. We now show how to do this with $A\neq A'$, which implies that FMSE enjoys $\sim$. 

Fix $(A,q) \in {\cal P}_0$, and for the case $A_\perp\not\equiv 0$ let $A' := A_\parallel - A_\perp$. Then $A\neq A'$, because $A_\perp \neq A'_\perp$; moreover, from $A_\parallel = A'_\parallel$ we get $A_{a\parallel} = A'_{a\parallel}$ and $A'_{s\parallel} = A_{s\parallel} \in H^{sp-s}$. Eventually, $|A'|^2 = |A'_\parallel|^2 + |A'_\perp|^2 = |A_\parallel|^2 + |-A_\perp|^2 = |A_\parallel|^2 + |A_\perp|^2 = |A|^2$ implies $\mathcal J_2A' = \mathcal J_2A$, and $A'$ verifies \emph{(p1)}. If instead we have $A_\perp\equiv 0$, let $A' = A_\parallel + RA_\parallel$, where $R$ is any $\pi/2$ rotation. Then as before $A_{a\parallel} = A'_{a\parallel}$ and $A'_{s\parallel} = A_{s\parallel} \in H^{sp-s}$, because $A_\parallel = A'_\parallel$. We also have $A\neq A'$, because $A_\perp = 0 \neq R A_\parallel = A'_\perp$. Finally, since $\mathcal J_2 A \in L^p$, $A'$ verifies \emph{(p1)}:
\spleqno{
\mathcal J_2A' & = \left( \int_{\mathbb R^n} |A'|^2 dy \right)^{1/2} = \left( \int_{\mathbb R^n} |A'_\parallel|^2 + |A'_\perp|^2 dy \right)^{1/2} \\ & = \left( \int_{\mathbb R^n} |A_\parallel|^2 + |RA_\parallel|^2 dy \right)^{1/2} = \left( \int_{\mathbb R^n} 2|A_\parallel|^2 dy \right)^{1/2} = \sqrt 2\, \mathcal J_2 A\;\qedhere.}\end{proof}

\begin{Lem}
FMSE does not have the gauge $\approx$.
\end{Lem}
\begin{proof}
Let $(A,q), (A',q') \in {\cal P}_0$ such that $(A,q)\approx(A',q')$. Then there exists $\phi\in G$ such that $(-\Delta)^s_A (u\phi) + qu\phi = \phi( (-\Delta)^s_{A'} u + q'u)$ for all $u\in H^s$. Fix $\psi \in C^{\infty}_c(\mathbb R^n)$, $x\in \mathbb R^n$ and $u(y):= \psi(y)e^{-1/|x-y|} |x-y|^{n+2s}$ as in Lemma 3.8. Then $u\in \cal S$, and by Lemma 3.3 and Remark 3.5, 
\begin{align*}
0 & = \mathcal C_{n,s} \; PV\int_{\mathbb R^n} \left( \sigma(x,y) \frac{u(x)\phi(x)-u(y)\phi(y)}{|x-y|^{n+2s}} - \sigma'(x,y)\frac{u(x)\phi(x)-u(y)\phi(x)}{|x-y|^{n+2s}} \right) \, dy \\ & \;\;\;\; +  u(x)\phi(x)(Q-Q') \\ & = \mathcal C_{n,s} \; PV\int_{\mathbb R^n} \frac{u(y)}{|x-y|^{n+2s}} \left( \sigma'(x,y)\phi(x)-\sigma(x,y) \phi(y) \right) \, dy \\ & = \mathcal C_{n,s} \int_{\mathbb R^n} \psi(y)e^{-1/|x-y|} \left( \sigma'(x,y)\phi(x)-\sigma(x,y) \phi(y) \right) \, dy\;.
\end{align*}

\noindent Here the principal value disappears because the integral is not singular. Given the arbitrarity of $\psi$ and the non negativity of the exponential, we deduce $\sigma(x,y)\phi(y) = \sigma'(x,y)\phi(x)$ for all $y\neq x$. On the other hand, since $\sigma, \sigma'$ are symmetric and $\phi>0$, by taking the symmetric part of each side  
\spleqno{
\sigma(x,y)\frac{\phi(x)+\phi(y)}{2} = ( \sigma(x,y)\phi(y) )_s = ( \sigma'(x,y)\phi(x) )_s = \sigma'(x,y)\frac{\phi(x)+\phi(y)}{2}\;.
}
\noindent This implies $\sigma=\sigma'$, and the equation can be rewritten as $\sigma(x,y) (\phi(y)-\phi(x))=0$. Being $\sigma>0$, it is clear that $\phi$ must be constant, and therefore equal to $1$. This means that whenever $(A,q), (A',q') \in {\cal P}_0$ are such that $(A,q)\approx(A',q')$ with some $\phi\in G$, then $\phi\equiv 1$, i.e. FMSE does not have the gauge $\approx$.  
\end{proof}

By the last two Lemmas, FMSE enjoys $\sim$, but not $\approx$. Observe that the reverse is true for the classical magnetic Schr\"odinger equation. This surprising difference is due to the non-local nature of the operators involved: FMSE has $\sim$ because the coefficient of its gradient term is not the whole vector potential $A$, as in the classical case, but just a part of it. On the other hand, the restriction imposed by the antisymmetry of such part motivates the absence of $\approx$.


\para{Bilinear form} Let $s\in(0,1),\; u,v\in H^s(\mathbb R^n)$, and define the \emph{bilinear form} $B^s_{A,q} : H^s \times H^s \rightarrow \mathbb R$ as follows:
\begin{equation*}
B^s_{A,q} [u,v] = \int_{\mathbb R^n}\int_{\mathbb R^n} \nabla^s_A u \cdot \nabla^s_A v \,dy dx + \int_{\mathbb R^n} quv \,dx\;.
\end{equation*}

\noindent Observe that by Fubini's theorem and Lemmas 3.3, 3.4
\begin{align*}
B^s_{A,q} [u,u] &= \langle(-\Delta)^s u, u\rangle + 2\langle \nabla^s u, A_{a\parallel}u \rangle + \langle Qu,u\rangle \\ & = \langle\nabla^s u, \nabla^su\rangle + 2\langle \nabla^s u, \alpha(\sigma-1)u \rangle + \langle Qu,u\rangle \\ & = \langle\nabla^s u, \nabla^su+ (\sigma-1)\alpha (u(x)-u(y))\rangle +\langle Qu,u\rangle \\ & = \langle\nabla^s u,\sigma\nabla^s u\rangle+\langle Qu,u\rangle \;.
\end{align*}

\noindent Since again by Lemma 3.4 we have $\sigma>1$, for the first term
$$ \langle\nabla^s u,\sigma\nabla^s u\rangle = \int_{\mathbb R^{2n}}\sigma |\nabla^s u|^2 \, dydx \geq \int_{\mathbb R^{2n}} |\nabla^s u|^2 \, dydx = \langle(-\Delta)^s u, u\rangle\;,$$
\noindent and thus $B^s_{A,q} [u,u] \geq B^s_{0,Q}[u,u]$. Now Lemma 2.6 from \cite{RS2017} gives the well-posedness of the direct problem for FMSE, in the assumption that $0$ is not an eigenvalue for the equation:  if $F \in (\tilde H^s(\Omega))^*$ then there exists a unique solution $u_F \in H^s(\Omega)$ to $B^s_{A,q}[u,v]=F(v), \; \forall v\in \tilde H^s(\Omega)$, that is a unique $u_F\in H^s(\Omega)$ such that $(-\Delta)^s_A u +qu =F$ in $\Omega$, $u_F|_{\Omega_e}=0$. For non-zero exterior value, see e.g. \cite{Co18} and \cite{GSU2017}; one also gets the estimate   
\begin{equation}\label{eq:estidirect}\|u_f\|_{H^s(\mathbb R^n)} \leq c(\|F\|_{(\tilde H^s(\Omega))^*} + \|f\|_{H^s(\mathbb R^n)})\;.\end{equation} 

\begin{Lem}\label{Lembil} Let $v,w\in H^s(\mathbb R^n)$, $f,g\in H^s(\Omega_e)$ and $u_f, u_g\in H^s(\mathbb R^n)$ be such that $((-\Delta)^s_A + q) u_f =((-\Delta)^s_A + q) u_g = 0$ in $\Omega$, $u_f|_{\Omega_e} = f$ and $u_g|_{\Omega_e} = g$. Then
\begin{enumerate}
\item $B^s_{A,q}[v,w] = B^s_{A,q}[w,v]\;$ (symmetry),
\item $|B^s_{A,q}[v,w]| \leq k\|v\|_{H^s(\mathbb R^n)}\|w\|_{H^s(\mathbb R^n)}\;$,
\item $B^s_{A,q}[u_f,e_g] = B^s_{A,q}[u_g,e_f]\;$,
\end{enumerate}
\noindent where $e_g, e_f \in H^s(\mathbb R^n)$ are extensions of $g,f$ respectively.
\end{Lem}

\begin{proof} Symmetry follows immediately from the definition. For the second point, use \emph{(e2)}, \emph{(e7)} and the definition of magnetic fractional gradient to write 
\spleqno{
|B^s_{A,q}[v,w]| & = |\langle \nabla^s_A v, \nabla^s_A w \rangle + \langle qv,w \rangle| \leq |\langle \nabla^s_A v, \nabla^s_A w \rangle| + |\langle qv,w \rangle| \\ & \leq \|\nabla^s_A v\|_{L^2} \|\nabla^s_A w\|_{L^2} + \|qv\|_{H^{-s}} \|w\|_{H^s} \\ & \leq k'\|v\|_{H^s} \|w\|_{H^s} + k''\|q\|_{L^{p}} \|v\|_{H^s} \|w\|_{H^s} \leq k \|v\|_{H^s} \|w\|_{H^s}\;.  
}
\noindent For the third point, first compute 
\spleqno{
B^s_{A,q}[u_f, u_g] & = \int_{\mathbb R^n} ((-\Delta)^s_A u_f + q u_f) u_g \,dx = \int_{\Omega_e} ((-\Delta)^s_A u_f + q u_f) u_g \,dx \\ & = \int_{\Omega_e} ((-\Delta)^s_A u_f + q u_f) e_g \,dx = B^s_{A,q}[u_f, e_g]\;,  }
\noindent and then $ B^s_{A,q}[u_f, e_g] = B^s_{A,q}[u_f, u_g] = B^s_{A,q}[u_g, u_f] = B^s_{A,q}[u_g, e_f]$.
\end{proof}


\para{The DN-map and the integral identity} 

\begin{Lem}\label{DNLem}
There exists a bounded, linear, self-adjoint map $\Lambda_{A,q}^s : X\rightarrow X^*$ defined by $$\langle \Lambda_{A,q}^s [f],[g] \rangle = B^s_{A,q}[u_f,g], \;\;\;\;\;\;\; \forall f,g\in H^s(\mathbb R^n)\; ,$$ \noindent where $X$ is the abstract quotient space $H^s(\mathbb R^n)/\tilde H^s(\Omega)$ and $u_f\in H^s(\mathbb R^n)$ solves $\mathrm (-\Delta)^s_A u_f + qu_f = 0$ in $\Omega$ with $u-f\in \tilde H^s(\Omega)$.
\end{Lem}  

\begin{proof}
We first prove that the tentative definition of the DN-map does not depend on the representatives of the equivalence classes involved. Let $\phi,\psi\in \tilde H^s(\Omega)$ and compute by Lemma \ref{Lembil} 
\spleqno{
B^s_{A,q}[u_{f+\phi}, g+\psi] & = \int_{\Omega_e} (g+\psi) ((-\Delta)^s_A + q)u_{f+\phi} \, dx \\ & = \int_{\Omega_e} g ((-\Delta)^s_A + q)u_{f} \, dx = B^s_{A,q}[u_{f}, g]\;.
}
\noindent The $\psi$ disappears because it vanishes in $\Omega_e$, while the $\phi$ plays actually no role, since $f=f+\phi$ over $\Omega_e$ implies $u_{f+\phi}=u_f$. The boundedness of $\Lambda_{A,q}^s$ follows from \ref{Lembil} and \eqref{eq:estidirect}: first compute 
\spleqno{
|\langle \Lambda_{A,q}^s [f],[g] \rangle| = |B^s_{A,q}[u_f,g]| \leq k\|u_f\|_{H^s}\|g\|_{H^s} \leq c\|f\|_{H^s}\|g\|_{H^s}\;,
}
\noindent for all $f\in [f],\;g\in [g]$, and then observe that this implies $$ |\langle \Lambda_{A,q}^s [f],[g] \rangle| \leq k \inf_{f\in [f]}\|f\|_{H^s}\inf_{g\in [g]}\|g\|_{H^s}= k\|[f]\|_X\|[g]\|_X\;. $$
\noindent Finally, we prove the self-adjointness using Lemma \ref{Lembil} again: $$ \langle \Lambda_{A,q}^s [f],[g] \rangle = B^s_{A,q} [u_f,e_g] = B^s_{A,q}[u_g,e_f] = \langle \Lambda_{A,q}^s [g],[f] \rangle = \langle [f],\Lambda_{A,q}^s [g] \rangle\;.\qedhere $$
\end{proof}
\vspace{2mm}

The DN-map will now be used to prove an integral identity. 

\begin{Lem}
Let $(A_1, q_1), (A_2, q_2) \in {\cal P}, \;f_1, f_2$ be exterior data belonging to $H^s(\mathbb R^n)$ and $u_i \in H^s(\mathbb R^n)$ be the solution of $(-\Delta)^s_{A_i} u_i + q_i u_i = 0$ with $u_i - f_i \in \tilde H^s(\Omega)$ for $i=1,2$. The following integral identity holds:
\spleq{intid}{
\langle (\Lambda^s_{A_1,q_1} & - \Lambda^s_{A_2,q_2}) f_1, f_2 \rangle = \\ & = 2\Big\langle \int_{\mathbb R^n} ((A_1)_{a\parallel} - (A_2)_{a\parallel})\cdot \nabla^su_1\, dy, u_2 \rangle + \langle (Q_1-Q_2)u_1, u_2 \Big\rangle\;.
} 
\end{Lem}

\begin{proof}
The proof is a computation based on the results of Lemmas \ref{DNLem} and 3.3:
\spleqno{
\langle (\Lambda^s_{A_1,q_1} & - \Lambda^s_{A_2,q_2}) f_1, f_2 \rangle = B^s_{A_1, q_1}[u_1,u_2] - B^s_{A_2,q_2}[u_1,u_2] \\ & = \langle \nabla^s u_1, \nabla^s u_2 \rangle + 2\Big\langle \int_{\mathbb R^n} (A_1)_{a\parallel} \cdot \nabla^su_1\, dy, u_2 \Big\rangle + \langle Q_1u_1, u_2 \rangle - \\ & \;\;\;\;\; - \langle \nabla^s u_1, \nabla^s u_2 \rangle - 2\Big\langle \int_{\mathbb R^n} (A_2)_{a\parallel} \cdot \nabla^su_1\, dy, u_2 \Big\rangle - \langle Q_2u_1, u_2 \rangle \\ & = 2\Big\langle \int_{\mathbb R^n} ((A_1)_{a\parallel} - (A_2)_{a\parallel})\cdot \nabla^su_1\, dy, u_2 \Big\rangle + \langle (Q_1-Q_2)u_1, u_2 \rangle\;.\qedhere }
\end{proof}


\para{The WUCP and the RAP} Let $W \subseteq \Omega_e$ be open and $u\in H^s(\mathbb R^n)$ be such that $u=0$ and $(-\Delta)^s_A u + qu = 0$ in $W$. If this implies that $u=0$ in $\Omega$ as well, we say that FMSE has got the WUCP. It is known that WUCP holds if both $A$ and $q$ vanish, that is, in the case of the fractional Laplace equation (see \cite{RS2017}).

Let $\mathcal R = \{ u_f|_{\Omega}, f \in C^\infty_c(W) \} \subset L^2(\Omega)$ be the set of the restrictions to $\Omega$ of those functions $u_f$ solving FMSE for some smooth exterior value $f$ supported in $W$. If $\cal R$ is dense in $L^2(\Omega)$, we say that FMSE has got the RAP.

\begin{Rem}
The WUCP and the RAP are non-local properties. For example, the RAP shows a certain freedom of the solutions to fractional PDEs, since it states that they can approximate any $L^2$ function. This is not the case for a local operator, e.g. the classical Laplacian, whose solutions are much more rigid.
\end{Rem}     

\begin{Lem}
The WUCP implies the RAP in the case of FMSE.
\end{Lem}

\begin{proof} We follow the spirit of the analogous Lemma of \cite{GSU2017}. Let $v\in L^2(\Omega)$, and assume that $\langle v,w\rangle =0$ for all $w\in \cal R$. Then if $f\in C^\infty_c(W)$ and $\phi\in \tilde H^s(\Omega)$ solves $(-\Delta)^s_A \phi + q\phi = v$ in $\Omega$, we have
\spleqno{
0 & = \langle v, u_f|_\Omega \rangle = \langle v, u_f - f \rangle = \int_{\mathbb R^n} v(u_f-f) \, dx \\ & = \int_{\Omega} v(u_f-f) \, dx = \int_{\Omega} ((-\Delta)^s_A \phi + q\phi )(u_f-f) \, dx \\ & = \int_{\mathbb R^n} ((-\Delta)^s_A \phi + q\phi )(u_f-f) \, dx \\ & = B^s_{A,q}[\phi, u_f] - \int_{\mathbb R^n} ((-\Delta)^s_A \phi + q\phi ) f \, dx \;.
}  
\noindent However, $B^s_{A,q}[\phi, u_f] = \int_{\mathbb R^n} ((-\Delta)^s_A u_f + qu_f ) \phi \, dx = 0$, and so $\int_{\mathbb R^n} ((-\Delta)^s_A \phi + q\phi ) f \, dx=0$. Given the arbitrarity of $f\in C^\infty_c(W)$, this implies that $(-\Delta)^s_A \phi + q\phi =0$ in $W$. Now we use the WUCP: from $(-\Delta)^s_A \phi + q\phi =0$ and $\phi=0$ in $W$, an open subset of $\Omega_e$, we deduce that $\phi=0$ in $\Omega$ as well. By the definition of $\phi$ and the fact that $v\in L^2(\Omega)$ it now follows that $v\equiv 0$. Thus if $\langle v,w\rangle =0$ holds for all $w\in \cal R$, then $v\in L^2(\Omega)$ must vanish; by the Hahn-Banach theorem this implies that $\cal R$ is dense in $L^2(\Omega)$. \end{proof}


\section{Main results}

\para{The inverse problem} We prove Theorem 1.1 under the assumption $(A,q)\in\cal P$, while for all the previous results we only required $(A,q)\in {\cal P}_0$. We find that \emph{(p5)} makes physical sense, as the random walk interpretation of FMSE suggests; however, we move the consideration of the general case to future work.   

\noindent By \emph{(p5)} and Lemma 3.5 we easily deduce that $\sigma(x,y)\equiv 1$ whenever $(x,y)\not\in\Omega^2$, since in this case $A_{a\parallel}(x,y)=0$. Another consequence of \emph{(p5)} is:

\begin{Lem}
Let $(A,q)\in\cal P$. Then FMSE enjoys the WUCP.
\end{Lem}
\begin{proof}
Suppose $W\subseteq\Omega_e$ is such that $u(x)=0$, $(-\Delta)^s_Au(x)+q(x)u(x) = 0$ when $x\in W$. Then $A_{a\parallel}(x,y)=0$, and by Lemma 3.3 $(-\Delta)^s u(x) =0$. Now the known WUCP for the fractional Laplacian (\cite{RS2017}) gives the result. \end{proof}

We are ready to solve the inverse problem, which we restate here:

\noindent \textbf{Theorem 1.1.} \emph{Let $\Omega \subset \mathbb R^n, \; n\geq 2$ be a bounded open set, $s\in(0,1)$, and let $(A_i,q_i) \in \cal P$ for $i=1,2$. Suppose $W_1, W_2\subset \Omega_e$ are open sets, and that the DN maps for the FMSEs in $\Omega$ relative to $(A_1,q_1)$ and $(A_2,q_2)$ satisfy   $$\Lambda^s_{A_1,q_1}[f]|_{W_2}=\Lambda^s_{A_2,q_2}[f]|_{W_2}, \;\;\;\;\; \forall f\in C^\infty_c(W_1)\;.$$ \noindent Then $(A_1,q_1)\sim(A_2,q_2)$, that is, the potentials coincide up to the gauge $\sim$. }

\begin{proof}
Without loss of generality, let $W_1 \cap W_2 = \emptyset$. Let $f_i \in C^\infty_c(W_i)$, and let $u_i \in H^s(\mathbb R^n)$ solve $(-\Delta)^s_{A_i} u_i + q_i u_i = 0$ with $u_i - f_i \in \tilde H^s(\Omega)$ for $i=1,2$. Knowing that the DN maps computed on $f\in C^\infty_c(W_1)$ coincide when restricted to $W_2$ and the integral identity \eqref{intid}, we write \emph{Alessandrini's identity}:
\spleq{Ale1}{ 0 & = \langle (\Lambda^s_{A_1,q_1} - \Lambda^s_{A_2,q_2}) f_1, f_2 \rangle \\ & = 2\Big\langle \int_{\mathbb R^n} ((A_1)_{a\parallel} - (A_2)_{a\parallel})\cdot \nabla^su_1\, dy, u_2 \Big\rangle + \langle (Q_1-Q_2)u_1, u_2 \rangle\;. }
  
\noindent We can refine \eqref{Ale1} by substituting every instance of $u_i$ with $u_i|_{\Omega}$. In fact, since $u_i$ is supported in $\Omega\cup W_i$ and $(\Omega\cup W_1)\cap(\Omega\cup W_2)=\Omega$,
\spleqno{ \langle (Q_1-Q_2)u_1, u_2 \rangle & = \int_{\mathbb R^n} u_1 u_2 (Q_1-Q_2)\; dx = \int_{\Omega} u_1 u_2 (Q_1-Q_2)\; dx \\ & = \int_{\Omega} u_1|_\Omega u_2|_\Omega (Q_1-Q_2)\; dx = \int_{\mathbb R^n} u_1|_\Omega u_2|_\Omega (Q_1-Q_2)\; dx .}

\noindent Moreover, by property \emph{(p5)}, 
\spleqno{
\Big\langle \int_{\mathbb R^n} & \nabla^su_1  \cdot ((A_1)_{a\parallel} - (A_2)_{a\parallel}) \, dy, u_2 \Big\rangle = \\ & = \int_{\mathbb R^n} u_2 \int_{\mathbb R^n} ((A_1)_{a\parallel} - (A_2)_{a\parallel})\cdot \nabla^su_1\, dy \, dx \\ & = \int_{\mathbb R^n} u_2(x) \int_{\mathbb R^n} (\sigma_1(x,y)-\sigma_2(x,y)) \,|\alpha|^2 (u_1(x)-u_1(y))\, dy \, dx \\ & = \int_{\Omega} (u_2|_\Omega)(x) \int_{\Omega} (\sigma_1(x,y)-\sigma_2(x,y)) \,|\alpha|^2 \Big((u_1|_\Omega)(x)-(u_1|_\Omega)(y)\Big)\, dy \, dx\;.
}

\noindent Eventually we get
\spleq{Ale2}{
0 & =2\int_{\mathbb R^n} (u_2|_\Omega)(x) \int_{\mathbb R^n} (\sigma_1(x,y)-\sigma_2(x,y)) \,|\alpha|^2 \Big((u_1|_\Omega)(x)-(u_1|_\Omega)(y)\Big)\, dy \, dx + \\& \;\;\; +  \int_{\mathbb R^n} u_1|_\Omega u_2|_\Omega (Q_1-Q_2)\; dx \;.
}

\noindent The RAP holds by Lemmas 3.18 and 3.19. Fix any $f\in L^2(\Omega)$, and let $f_i^{(k)} \in C^\infty_c(W_i)$ for $i=1,2$ and $k\in \mathbb N$ be such that $u_1^{(k)}|_\Omega\rightarrow 1$, $u_2^{(k)}|_\Omega\rightarrow f$ in $L^2$. Inserting these solutions in \eqref{Ale2} and taking the limit as $k\rightarrow\infty$ implies that $\int_{\mathbb R^n} f (Q_1-Q_2)\; dx =0$, so that, given that $f\in L^2(\Omega)$ is arbitrary, we deduce $Q_1(x) = Q_2(x)$ for $x\in\Omega$. Coming back to \eqref{Ale2}, we can write
\spleqno{ \int_{\mathbb R^n} (u_2|_\Omega)(x) & \int_{\mathbb R^n} (\sigma_1(x,y)-\sigma_2(x,y)) \frac{(u_1|_\Omega)(x)-(u_1|_\Omega)(y)}{|x-y|^{n+2s}}\, dy \, dx = 0, }

\noindent where $u_i \in H^s(\mathbb R^n)$ once again solves $(-\Delta)^s_{A_i} u_i + q_i u_i = 0$ with $u_i - f_i \in \tilde H^s(\Omega)$ for some $f_i \in C^\infty_c(W_i)$ and $i=1,2$. Choosing $u_2^{(k)}|_\Omega\rightarrow f$ in $L^2$ for some arbitrary $f\in L^2$, by the same argument 
$$ \int_{\mathbb R^n} (\sigma_1(x,y)-\sigma_2(x,y)) \frac{(u_1|_\Omega)(x)-(u_1|_\Omega)(y)}{|x-y|^{n+2s}}\, dy = 0 $$ 
	
\noindent for $x\in \Omega$. Fix now some $x \in \Omega$ and an arbitrary $\psi\in C^\infty_c(\Omega)$. Since $g(y):= \psi(y)e^{-1/|x-y|}|x - y|^{n+2s} \in {\cal S} \subset L^2(\Omega)$ as in Lemma \ref{charofnewg}, by the RAP we find a sequence $u_1^{(k)}|_\Omega \rightarrow g$. Substituting these solutions and taking the limit, 

$$ \int_{\mathbb R^n} (\sigma_1(x,y)-\sigma_2(x,y))\psi(y)e^{-1/|x-y|}\, dy = 0\;. $$

\noindent Thus we conclude that for all $x\in \Omega$ it must be $\sigma_1(x,y)=\sigma_2(x,y)$ for all $y\in\Omega$, i.e. $\sigma_1 = \sigma_2$ over $\Omega^2$. But then $\sigma_1$ and $\sigma_2$ coincide everywhere, because they are both 1 in $\mathbb R^{2n}\setminus \Omega^2$. This means that $(A_1)_{a\parallel}=(A_2)_{a\parallel}$. Moreover, since by \emph{(p2)}, \emph{(p4)} and \emph{(p5)} we have $Q_1 = 0 = Q_2$ over $\Omega_e$, by the argument above $Q_1 = Q_2$ everywhere. It thus follows from Lemma \ref{charofnewg} that $(A_1,q_1)\sim(A_2,q_2)$. 
\end{proof}


\section{A random walk interpretation for FMSE}

Diffusion phenomena can often be seen as continuous limits of random walks. The classical result for the Laplacian was extended in \cite{valdi} to the fractional one by considering long jumps. Similarly, the fractional conductivity equation was shown in \cite{Co18} to arise from a long jump random walk with weight $\gamma^{1/2}$, where $\gamma$ is the conductivity. We now show how the leading term in FMSE is itself the limit of a long jump random walk with weights. For simplicity, here we take $\sigma$ as smooth and regular as needed. Let $h>0,\; \tau=h^{2s},  \; k\in \mathbb{Z}^n$, $x\in h\mathbb{Z}^n$ and $t\in \tau\mathbb{Z}$. We consider a random walk on $h\mathbb{Z}^n$ with time steps from $\tau\mathbb{Z}$. Define
 \begin{equation*}
  f(x,k) :=
  \begin{cases}
    \sigma(x,x+hk) |k|^{-n-2s}       & \mbox{if} \quad k\neq 0 \\
    0  & \mbox{if} \quad k=0
  \end{cases}\;,
	\end{equation*}

\noindent and then observe that $\forall x\in h\mathbb{Z}^n$
\begin{equation*}
\begin{split}
\sum_{k\in\mathbb{Z}^n} f(x,k) & = \sum_{k\in\mathbb{Z}^n \setminus \{0\}} f(x,k) = \sum_{k\in\mathbb{Z}^n \setminus \{0\}} \sigma(x,x+hk) |k|^{-n-2s} \\ & \leq \|\sigma\|_{L^\infty} \sum_{k\in\mathbb{Z}^n \setminus \{0\}} |k|^{-n-2s} < \infty\;.
\end{split}
\end{equation*}
\noindent Thus we can normalize $f(x,k)$, and get the new function $P(x,k)$
\begin{equation*}
  P(x,k) :=
  \begin{cases}
    \left( \sum_{j\in\mathbb{Z}^n} f(x,j) \right)^{-1} \sigma(x,x+hk) |k|^{-n-2s}       & \mbox{if} \quad k\neq 0 \\
    0  & \mbox{if} \quad k=0
  \end{cases}\;. 
	\end{equation*}
\noindent $P(x,k)$ takes values in $[0,1]$ and verifies $\sum_{k\in\mathbb{Z}^n} P(x,k)=1$; we interpret it as the probability that a particle will jump from $x+hk$ to $x$ in the next step. 

\begin{Rem}
Let us compare $P(x,k)$ for the fractional Laplacian, conductivity and magnetic Laplacian operators. $P(x,k)$ always decreases when $k$ increases; the fractional Laplacian, which has $\sigma(x,y)\equiv 1$, treats all the points of $\mathbb R^n$ equally: no point is intrinsically more likely to be reached at the next jump; the fractional conductivity operator, which has $\sigma(x,y) = \sqrt{\gamma(x)\gamma(y)}$, distinguishes the points of $\mathbb R^n$: those with high conductivity are more likely to be reached. However, the conductivity field is independent from the current position of the particle. The magnetic fractional Laplacian operator has no special $\sigma(x,y)$ and it distinguishes the points of $\mathbb R^n$ in a more subtle way, as the conductivity field depends on the position of the particle: the same point may have high conductivity if the particle is at $x$ and a low one if it is at $y$.   
\end{Rem}

\begin{Rem}
We now see why $\sigma>0$ and $\sigma(x,y)=1$ if $(x,y)\not\in\Omega^2$: these are needed for $y\mapsto \sigma(x,y)$ to be a conductivity as in \cite{Co18} for all $x\in\mathbb R^n$. 
\end{Rem}

\noindent Let $u(x,t)$ be the probability that the particle is at point $x$ at time $t$. Then \begin{equation*} u(x, t+\tau) = \sum_{k\in\mathbb{Z}^n\setminus\{0\}}P(x,k)u(x+hk,t) \;\;.  \end{equation*}

\noindent We can compute $\partial_t u(x,t)$ as the limit for $\tau\rightarrow 0$ of the difference quotients, and then substitute the above formula (see \cite{Co18}). As the resulting sum approximates the Riemannian integral, we eventually get that for some constant $C>0$
$$ \partial_t u(x,t) = C \int_{\mathbb R^n} \sigma(x,y) \frac{u(y,t)-u(x,y)}{|x-y|^{n+2s}}\, dy\;. $$
\noindent If $u(x,t)$ is independent of $t$, the leading term of FMSE is recovered. 


\section{One slight generalization}

\noindent We now briefly consider a fractional magnetic \emph{conductivity} equation (FMCE) and show that it shares similar features as FMSE. Let $(A,q)\in \cal P$ and let $\gamma$ be a conductivity in the sense of \cite{Co18}. Consider $u\in H^s(\mathbb R^n)$. Since $\nabla^s_A: H^s(\mathbb R^n) \rightarrow L^2(\mathbb R^{2n})$, if $\Theta(x,y):= \sqrt{\gamma(x)\gamma(y)}$Id by the properties of $\gamma$ we know that $\Theta \cdot\nabla^s_A u \in L^2(\mathbb R^{2n})$. Thus we define the \emph{fractional magnetic conductivity operator} 
$$ \mathrm C^s_{\gamma,A} u(x) := (\nabla\cdot)^s_A(\Theta\cdot\nabla^s_A u)(x)\;,\;\;\;\;\;\;\;\mathrm C^s_{\gamma,A} : H^s(\mathbb R^n) \rightarrow H^{-s}(\mathbb R^n)\;.  $$
\noindent We say that $u\in H^s(\mathbb R^n)$ solves the FMCE with exterior value $f\in H^s(\Omega_e)$ if 
$$ \left\{\begin{array}{lr}
        \mathrm C^s_{\gamma,A} u(x) + q(x)u(x) = 0 & \text{in } \Omega\,\;\\
        u=f & \text{in } \Omega_e
        \end{array}\right.$$

\noindent holds in weak sense.

\begin{Lem}
Let $u\in H^s(\mathbb R^n)$, $g\in H^s(\Omega_e)$, $w=\gamma^{1/2}u$ and $f=\gamma^{1/2}g$. Moreover, let $(A,q)\in {\cal P}$ and 
\spleqno{q' := q'_{\gamma,A,q} & = \frac{q}{\gamma} - (\nabla\cdot)^s A_{s\parallel} + \frac{(\nabla\cdot)^s(A \gamma^{1/2}(y))}{\gamma^{1/2}(x)} -\frac{(-\Delta)^s(\gamma^{1/2})}{\gamma^{1/2}(x)} +\\ & + \int_{\mathbb R^n} \Big( - \frac{\nabla^s(\gamma^{1/2}) \cdot A}{\gamma^{1/2}(x)} + |A|^2 \Big(\frac{\gamma^{1/2}(y)}{\gamma^{1/2}(x)}-1\Big) \Big) \,dy\;.} 

\noindent FMCE with potentials $(A,q)$, conductivity $\gamma$ and exterior value $g$ is solved by $u$ if and only if $w$ solves FMSE with potentials $(A,q')$ and exterior value $f$, i.e.
\begin{equation*} 
\left\{\begin{array}{lr}
        \mathrm C^s_{\gamma,A} u + qu = 0 & \text{in } \Omega\,\;\\
        u=g & \text{in } \Omega_e
        \end{array}\right. \;\;\Leftrightarrow\;\; 
		\left\{\begin{array}{lr}
        (-\Delta)^s_A w +q'w=0 & \text{in } \Omega\,\;\\
        w=f & \text{in } \Omega_e
        \end{array}\right. \;.		 
\end{equation*}

\noindent \emph{Moreover, the following formula holds for all $w\in H^s(\mathbb R^n)$:}
$$ \mathrm C^s_{\gamma,A} (\gamma^{-1/2}w) + q\gamma^{-1/2}w = \gamma^{1/2} \Big((-\Delta)^s_A +q'\Big)w\,. $$
\end{Lem} 

\begin{proof}
Let us start from some preliminary computations. One sees that
\spleqno{
\nabla^s w & = \nabla^s (\gamma^{1/2}u) = \nabla^s u + \nabla^s(mu) = \nabla^s u + m(y)\nabla^s u + u(x)\nabla^s m \\ & = \gamma^{1/2}(y) \nabla^s u + u(x)\nabla^s (\gamma^{1/2}) = \gamma^{1/2}(y) \nabla^s u + w(x)\frac{\nabla^s (\gamma^{1/2})}{\gamma^{1/2}(x)}\;, 
} \noindent from which $\nabla^s u = \frac{\nabla^s w}{\gamma^{1/2}(y)}- w(x)\frac{\nabla^s(\gamma^{1/2})}{\gamma^{1/2}(x)\gamma^{1/2}(y)}$, and eventually
\spleq{prel1}{
\nabla^s_A u = \frac{\nabla^s w}{\gamma^{1/2}(y)}- w(x)\frac{\nabla^s(\gamma^{1/2})}{\gamma^{1/2}(x)\gamma^{1/2}(y)} + A(x,y)\frac{w(x)}{\gamma^{1/2}(x)}\;.
}
\noindent By the definition of magnetic fractional divergence, if $v\in H^s(\mathbb R^n)$,
\spleqno{
\langle (\nabla\cdot)^s_A&(\Theta\cdot\nabla^s_A u),v \rangle = \langle \gamma^{1/2}(x)\gamma^{1/2}(y)\nabla^s_A u, \nabla^s_A v \rangle
\\ & = \langle \gamma^{1/2}(x)\gamma^{1/2}(y)\nabla^s_A u, \nabla^s v \rangle + \langle \gamma^{1/2}(x)\gamma^{1/2}(y)\nabla^s_A u, A v \rangle \\ & = \langle \gamma^{1/2}(x)\gamma^{1/2}(y)\nabla^s_A u, \nabla^s v \rangle + \Big\langle \int_{\mathbb R^n}\gamma^{1/2}(y)\nabla^s_A u \cdot A \,dy, \gamma^{1/2} v \Big\rangle\;.
}
\noindent Applying formula \eqref{prel1}, we get
\begin{align}
\langle (\nabla&\cdot)^s_A(\Theta\cdot\nabla^s_A u),v \rangle = \langle \gamma^{1/2}(x)\nabla^s w, \nabla^s v \rangle
+ \langle w(x) (A(x,y) \gamma^{1/2}(y) -\nabla^s(\gamma^{1/2}) ) , \nabla^s v \rangle \nonumber\\ & \;\;\;\;\; + \Big\langle \int_{\mathbb R^n}\gamma^{1/2}(y)\Big( \frac{\nabla^s w}{\gamma^{1/2}(y)}- w(x)\frac{\nabla^s(\gamma^{1/2})}{\gamma^{1/2}(x)\gamma^{1/2}(y)} + A(x,y)\frac{w(x)}{\gamma^{1/2}(x)} \Big) \cdot A \,dy, \gamma^{1/2} v \Big\rangle \; \nonumber\\ & = \langle \gamma^{1/2}(x)\nabla^s w, \nabla^s v \rangle
+ \langle w(x) (A(x,y) \gamma^{1/2}(y) -\nabla^s(\gamma^{1/2}) ) , \nabla^s v \rangle \label{mainpart} \\ & \;\;\;\;\; + \Big\langle \int_{\mathbb R^n} \Big( \nabla^s w \cdot A - w(x)\frac{\nabla^s(\gamma^{1/2}) \cdot A}{\gamma^{1/2}(x)} + |A|^2 w(x) \frac{\gamma^{1/2}(y)}{\gamma^{1/2}(x)} \Big) \,dy, \gamma^{1/2} v \Big\rangle\;.\nonumber\end{align}

\noindent We treat the resulting terms separately. For the first one, by symmetry,
\begin{align}
\langle \gamma^{1/2}&(x)\nabla^s w, \nabla^s v \rangle  = \langle \nabla^s w, \gamma^{1/2}(x)\nabla^s v \rangle = \langle \nabla^s w, \nabla^s(v\gamma^{1/2})-v(y)\nabla^s (\gamma^{1/2}) \rangle \nonumber \\ & = \langle (-\Delta)^s w, v\gamma^{1/2} \rangle - \langle \nabla^s w, v(y) \nabla^s(\gamma^{1/2}) \rangle = \langle (-\Delta)^s w, v\gamma^{1/2} \rangle - \langle \nabla^s w, v(x) \nabla^s(\gamma^{1/2}) \rangle \nonumber \\ & = \langle (-\Delta)^s w, v\gamma^{1/2} \rangle - \Big\langle \int_{\mathbb R^n} \nabla^s w\cdot \frac{\nabla^s(\gamma^{1/2})}{\gamma^{1/2}(x)}\, dy, \gamma^{1/2}v \Big\rangle\;. \label{firstpart}
\end{align}

\noindent For the second part of \eqref{mainpart}, we will compute as follows:
\begin{align}
& \langle A(x,y) \gamma^{1/2}(y) -\nabla^s(\gamma^{1/2})  , w(x)\nabla^s v \rangle = \nonumber \\ & =  \langle A(x,y) \gamma^{1/2}(y) -\nabla^s(\gamma^{1/2})  , \nabla^s (vw) - v(y)\nabla^s w \rangle \nonumber \\ & = \Big\langle (\nabla\cdot)^s\Big(A(x,y) \gamma^{1/2}(y) -\nabla^s(\gamma^{1/2}))  , vw \Big\rangle  - \Big\langle \Big(A(x,y) \gamma^{1/2}(y) -\nabla^s(\gamma^{1/2})\Big) v(y)  , \nabla^s w \Big\rangle \nonumber \\ & = \Big\langle \Big( \frac{(\nabla\cdot)^s(A \gamma^{1/2}(y))}{\gamma^{1/2}(x)} -\frac{(-\Delta)^s(\gamma^{1/2})}{\gamma^{1/2}(x)}  \Big) w(x) , v\gamma^{1/2} \Big\rangle - \Big\langle \Big(A(y,x) \gamma^{1/2}(x) -\nabla^s(\gamma^{1/2})\Big) v(x)  , \nabla^s w \Big\rangle \nonumber \\ & = \Big\langle \Big( \frac{(\nabla\cdot)^s(A \gamma^{1/2}(y))}{\gamma^{1/2}(x)} -\frac{(-\Delta)^s(\gamma^{1/2})}{\gamma^{1/2}(x)}  \Big) w(x) , v\gamma^{1/2} \Big\rangle - \label{secondpart} \\ & \;\;\;\;\; - \Big\langle \int_{\mathbb R^n}A(y,x) \cdot \nabla^s w\, dy  , v \gamma^{1/2} \Big\rangle + \Big\langle  \int_{\mathbb R^n} \frac{\nabla^s(\gamma^{1/2})}{\gamma^{1/2}(x)} \cdot \nabla^s w \, dy  , v\gamma^{1/2} \Big\rangle\;.\nonumber
\end{align}
\noindent Substituting \eqref{firstpart} and \eqref{secondpart} into \eqref{mainpart}, we conclude the proof:
\begin{align*} \langle &(\nabla\cdot)^s_A (\Theta\cdot\nabla^s_A u),v \rangle = \langle (-\Delta)^s w, v\gamma^{1/2} \rangle - \Big\langle \int_{\mathbb R^n} \nabla^s w\cdot \frac{\nabla^s(\gamma^{1/2})}{\gamma^{1/2}(x)}\, dy, \gamma^{1/2}v \Big\rangle + \\ & \;\;\;\;\; + \Big\langle \Big( \frac{(\nabla\cdot)^s(A \gamma^{1/2}(y))}{\gamma^{1/2}(x)} -\frac{(-\Delta)^s(\gamma^{1/2})}{\gamma^{1/2}(x)}  \Big) w(x) , v\gamma^{1/2} \Big\rangle - \\ & \;\;\;\;\; - \Big\langle \int_{\mathbb R^n}A(y,x) \cdot \nabla^s w\, dy  , v \gamma^{1/2} \Big\rangle + \Big\langle  \int_{\mathbb R^n} \frac{\nabla^s(\gamma^{1/2})}{\gamma^{1/2}(x)} \cdot \nabla^s w \, dy  , v\gamma^{1/2} \Big\rangle + \\ & \;\;\;\;\; + \Big\langle \int_{\mathbb R^n} \Big( \nabla^s w \cdot A - w(x)\frac{\nabla^s(\gamma^{1/2}) \cdot A}{\gamma^{1/2}(x)} + |A|^2 w(x) \frac{\gamma^{1/2}(y)}{\gamma^{1/2}(x)} \Big) \,dy, \gamma^{1/2} v \Big\rangle \\ & = \Big\langle (-\Delta)^s w + 2 \int_{\mathbb R^n}A_{a\parallel} \cdot \nabla^s w\, dy  + w(x)\Big(\int_{\mathbb R^n} |A|^2\, dy + (\nabla\cdot)^s A_{s\parallel} \Big) , v\gamma^{1/2} \Big\rangle + \\ & \;\;\;\;\; + \Big\langle \left\{ - (\nabla\cdot)^s A_{s\parallel} + \frac{(\nabla\cdot)^s(A \gamma^{1/2}(y))}{\gamma^{1/2}(x)} -\frac{(-\Delta)^s(\gamma^{1/2})}{\gamma^{1/2}(x)}+ \right.  \\ & \;\;\;\;\; \;\;\;\;\;\left.+ \int_{\mathbb R^n} \Big( - \frac{\nabla^s(\gamma^{1/2}) \cdot A}{\gamma^{1/2}(x)} + |A|^2 \Big(\frac{\gamma^{1/2}(y)}{\gamma^{1/2}(x)}-1\Big) \Big) \,dy \right\} w(x) , v\gamma^{1/2} \Big\rangle \\ & = \langle (-\Delta)^s_A w +     (q'-q/\gamma)w, v\gamma^{1/2} \rangle\;. \qedhere\end{align*}\end{proof} 

\noindent Thus the FMCEs can be reduced to FMSEs; hence, we know that FMCE enjoys the same gauges as FMSE, and most importantly we can consider and solve an analogous inverse problem. 
\vspace{3mm}

\noindent {\bf Acknowledgements.} This work is part of the PhD research of the author, who was partially supported by the European Research Council under Horizon 2020 (ERC CoG 770924). The author wishes to express his sincere gratitude to Professor Mikko Salo for his reliable guidance and constructive discussion in the making of this work.

\bibliography{Biblio}

\end{document}